\documentclass[12pt]{amsart}
\usepackage{amsmath,amsthm,amssymb} 
\usepackage{color, a4}
\usepackage{comment}

\newtheorem{theorem}{Theorem}[section]
\newtheorem{Prop}[theorem]{Proposition}
\newtheorem{Thm}[theorem]{Theorem}
\newtheorem{Lem}[theorem]{Lemma}

\newtheorem{Con}[theorem]{Conjecture}

\theoremstyle{definition}

\newtheorem{Def}[theorem]{Definition}
\newtheorem{Ex}[theorem]{Example}
\theoremstyle{remark}

\numberwithin{equation}{section}

\newcommand{\Z}{{\mathbb Z}}
\newcommand{\In}{\mbox{\rm Inn}}
\newcommand{\Aut}{\mbox{\rm Aut}}
\newcommand{\id}{\mbox{\rm id}}

\begin{document}

\title{On classification of quandles of cyclic type} 

\author{Seiichi Kamada}
\address[S.~Kamada]{Department of Mathematics, Osaka City University, 
Osaka 558-8585, Japan} 
\email{skamada@sci.osaka-cu.ac.jp} 

\author{Hiroshi Tamaru}
\address[H.~Tamaru]{Department of Mathematics, Hiroshima University, 
Higashi-Hiroshima 739-8526, Japan}
\email{tamaru@math.sci.hiroshima-u.ac.jp}

\author{Koshiro Wada} 
\address[K.~Wada]{Department of Mathematics, Hiroshima University, 
Higashi-Hiroshima 739-8526, Japan}
\email{d126092@hiroshima-u.ac.jp}

\thanks{
The first author was partially supported by KAKENHI (21340015, 23654027). 
The second author was partially supported by KAKENHI (24654012).} 

\begin{abstract}
In this paper, we study quandles of cyclic type, 
which form a particular subclass of finite quandles. 
The main result of this paper describes the set of 
isomorphism classes of quandles 
of cyclic type in terms of certain cyclic permutations. 
By using our description, we 
give a direct classification of 
quandles of cyclic type with cardinality up to $12$. 
\end{abstract}

\maketitle

\section{Introduction}

The notion of quandle was introduced by Joyce (\cite{Joyce}) 
as a set with a binary operator, 
satisfying three axioms corresponding to 
Reidemeister moves of a classical knot. 
In knot theory, quandles play a lot of important roles, 
and have provided several invariants of knots 
(\cite{CJKLS, FennRourke, IIJO, KamadaOshiro, Nosaka}). 
For further information, 
we refer to \cite{Carter, Kamada} and references therein. 
Among others, 
Carter, Jelsovsky, the first author, 
Langford and Saito (\cite{CJKLS}) gave 
strong invariants, called quandle cocycle invariants, defined by quandle cocycles. 
For example, 
they gave a $3$-cocycle of the dihedral quandle $R_3$ with cardinality $3$, 
and apply it to prove the non-invertibility of the 2-twist span trefoil. 

Quandles provide several invariants of knots, 
but on the other hand, 
it is difficult to calculate these invariants explicitly, 
especially if the structure of the quandle is complicated. 
Therefore, it is of importance to study special classes of quandles, 
whose quandle structures are easy to handle. 
From this point of view, 
we study quandles of cyclic type, 
whose name was introduced 
in \cite{Tamaru}. 
A quandle with cardinality $n$ is said to be of \textit{cyclic type} 
if all right multiplications are cyclic permutations of order $n-1$. 
Since this quandle structure is very tractable, 
quandles of cyclic type are potentially useful for applications in knot theory. 

We here recall some known results on quandles of cyclic type. 
In \cite{Lopes-Roseman}, 
Lopes and Roseman essentially studied quandles of cyclic type, 
which they call 
quandles with constant profile $(\{1, n-1 \}, \dots, \{1, n-1 \})$. 
They studied such quandles in terms of cyclic permutations, 
and classified those with cardinality up to $8$. 
Subsequently, Hayashi (\cite{Hayashi}) studied 
the structures of quandles of cyclic type, 
and gave a table of those with cardinality up to $35$. 
Note that his table is obtained by using the list of connected quandles 
with cardinality up to $35$ 
(called Vendramin's list \cite{Vendramin}). 
Independently, the second author (\cite{Tamaru}) studied quandles of cyclic type, 
and classified those with prime cardinality. 
In particular, for every prime number $p \geq 3$, 
there exists a quandle of cyclic type with cardinality $p$. 
This suggests that the class of quandles of cyclic type is fruitful. 

In this paper, we study and describe the set $C_n$ of 
isomorphism classes of quandles of cyclic type with cardinality $n$. 
In fact, our main theorem gives a bijection from $C_n$ onto $F_{n}$, 
where $F_n$ 
denotes the set of cyclic permutations of order $n-1$ satisfying two conditions. 
This bijection is useful for studying quandles of cyclic type, 
since such quandles can be characterized by certain cyclic permutations. 
We then apply our main theorem to the classification of quandles of cyclic type, 
and provide a list of those with cardinality up to $12$.  
Our study extends some of the results by Lopes and Roseman (\cite{Lopes-Roseman}). 
In fact, they also studied cyclic permutations 
determined by quandles of cyclic type, which are similar to ours. 
Our new contribution is to show that it gives a well-defined and bijective map. 
Furthermore, our argument 
gives a direct and classification-free 
proof for a part of the table given by Hayashi (\cite{Hayashi}). 

This paper is organized as follows. 
In Section~2 
we recall some fundamental notions on quandles. 
In Section~3,  
the definition and some properties of quandles of cyclic type are summarized. 
We state the main theorem in Section~4, 
and give a table of quandles of cyclic type with cardinality up to $12$. 
Section~5 
contains the proof of the main theorem. 

The authors would like to express our gratitude to Chuichiro Hayashi 
for valuable comments, 
which lead us to \cite{Hayashi, Lopes-Roseman}.

\section{Preliminaries for quandles}

In this section we recall some fundamental notions on quandles.  

\begin{Def}
\label{def:quandle}
Let $X$ be a set and $\ast : X \times X \to X$ 
be a binary operator. 
The pair $(X , \ast)$ is called a 
\textit{quandle}
if 
\begin{enumerate}
\item[(Q1)]
$\forall x \in X$, $x \ast x = x$, 
\item[(Q2)]
$\forall x , y \in X$, 
$\exists! z \in X$ 
: $z \ast y = x$, and 
\item[(Q3)]
$\forall x , y , z \in X$, 
$(x \ast y) \ast z = (x \ast z) \ast (y \ast z)$. 
\end{enumerate}
\end{Def}

If $(X , \ast)$ is a quandle, 
then $\ast$ is called a \textit{quandle structure} on $X$. 
We restate the definition of a quandle
as follows. 

\begin{Prop}[\cite{FennRourke, Tamaru}]
\label{prop:s-quandle-structure}
Let $X$ be a set, and 
assume that there exists a map $s_x : X \to X$ for every $x \in X$. 
Then, the binary operator $\ast$ defined by 
$y \ast x := s_x(y)$ 
is a quandle structure on $X$ if and only if 
\begin{enumerate}
\item[(S1)]
$\forall x \in X$, $s_x(x) = x$, 
\item[(S2)]
$\forall x \in X$, $s_x$ is bijective, and 
\item[(S3)]
$\forall x , y \in X$, $s_x \circ s_y = s_{s_x(y)} \circ s_x$. 
\end{enumerate}
\end{Prop}

Instead of Definition~\ref{def:quandle}, 
throughout this paper, we denote the quandle by $X = (X,s)$ 
with the quandle structure 
\begin{align*}
s : X \to \mathrm{Map} (X,X) : x \mapsto s_x . 
\end{align*}
Here $\mathrm{Map} (X,X)$ denotes the set of all maps from $X$ to $X$.

\begin{Ex}\label{ex:1} 
The following $(X, s)$ are quandles$:$ 
\begin{enumerate}
\item[(1)] 
Let $X$ be any set and $s_x := \id_X$ for every $x \in X$. 
Then the pair $(X, s)$ is called the \textit{trivial quandle}. 

\item[(2)] 
Let $X := \{1, \dots, n\}$ 
and $s_i(j) := 2i-j$ ($\mathrm{mod}\ n$) for any $i, j\in X$. 
Then the pair $(X, s)$ is called the 
\textit{dihedral quandle} with cardinality $n$. 

\item[(3)] 
Let $X := \{ 1, 2, 3, 4 \}$ and 
$$
s_1:=(234), \ s_2:=(143), \ s_3:=(124), \ s_4:=(132).
$$
Then the pair $(X, s)$ is called the \textit{tetrahedron quandle}.   
\end{enumerate}
\end{Ex}

Note that $(234)$, $(143)$, and so on, denote the cyclic permutations. 
We use this symbol frequently in the later sections.

\begin{Def}\label{def:iso}
Let $(X,s^X)$, $(Y,s^Y)$ be quandles, and $f : X \to Y$ be a map. 
\begin{enumerate}
\item
$f$ is called a 
\textit{homomorphism} 
if  for every $x \in X$, 
$f \circ s^X_x = s^Y_{f(x)} \circ f$ holds. 
\item
$f$ is called an 
\textit{isomorphism} 
if $f$ is a bijective homomorphism. 
\end{enumerate}
\end{Def}

An isomorphism from a quandle $(X, s)$ 
onto itself is called an \textit{automorphism}. 
The set of automorphisms of $(X, s)$ forms a group, 
which is called the 
\textit{automorphism group} and denoted by $\Aut (X, s)$. 

Note that $s_x$ $(x \in X)$ is an automorphism of $(X,s)$.   
The subgroup of $\Aut (X, s)$ generated by 
$\{s_x \mid x \in X\}$ is called the 
\textit{inner automorphism group} 
of $(X, s)$ and denoted by $\In(X, s)$. 

\begin{Def}\label{def:connected} 
A quandle $(X, s)$ is said to be \textit{connected} 
if $\In(X,s)$ acts transitively on $X$. 
\end{Def}

On the connectedness of the quandles in Example~\ref{ex:1}, 
the following is well-known. 
We denote by $\# X$ the cardinality of $X$.

\begin{Ex}
One has the following: 
\begin{enumerate}
\item
The trivial quandle $(X,s)$ is connected if and only if $\# X = 1$. 
\item
The dihedral quandle $(X,s)$ is connected if and only if $\# X$ is odd. 
\item
The tetrahedron quandle is connected. 
\end{enumerate}
\end{Ex}

\section{Quandles of cyclic type} 

\label{section:c-type}

From now on we always assume that 
a quandle $X = (X,s)$ is finite and satisfies $\# X \geq 3$. 
In this section, we recall the definition and some properties 
of quandles of cyclic type given in 
\cite{Tamaru}. 

\begin{Def}[\cite{Tamaru}]
A quandle $(X,s)$ with $\# X = n \geq 3$ is said to be of 
\textit{cyclic type} 
if for every $x \in X$, 
$s_x$ acts on $X \setminus \{ x \}$ as a cyclic permutation of order $n-1$. 
\end{Def}

This notion is closely related to the notion of two-point homogeneous quandle. 
A quandle $(X, s)$ is said to be 
\textit{two-point homogeneous} 
if for any $(x_1, x_2), (y_1,y_2) \in X \times X$ 
satisfying $x_1 \neq x_2$ and $y_1 \neq y_2$, 
there exists $f \in \In (X, s)$ such that $(f(x_1), f(x_2)) = (y_1, y_2)$. 
The second author studied quandles of cyclic type in \cite{Tamaru} 
because of the following proposition. 

\begin{Prop}[\cite{Tamaru}]
\label{prop:c-TPH}
Every quandle of cyclic type is two-point homogeneous. 
\end{Prop}

The following is a characterization of quandles of cyclic type, 
which we use in the latter arguments. 
In particular, quandles of cyclic type must be connected. 

\begin{Prop}[\cite{Tamaru}]
\label{prop:cyclic-isotropic}
Let $X = (X,s)$ be a quandle with $\# X = n \geq 3$. 
Then, $X$ is of cyclic type if and only if 
\begin{enumerate}
\item[(i)]
$X$ is connected, and 
\item[(ii)]
there exists $x \in X$ such that 
$s_x$ acts on $X \setminus \{ x \}$ as a cyclic permutation of order $n-1$. 
\end{enumerate}
\end{Prop}

If the structure of a quandle is given, 
then one can easily check whether it is of cyclic type or not. 
We here give some easy examples. 

\begin{Ex}
One has the following$:$ 
\begin{enumerate}
\item
The trivial quandles are not of cyclic type.  
\item
The dihedral quandle $(X,s)$ is of cyclic type if and only if $\# X = 3$. 
\item
The tetrahedron quandle is of cyclic type. 
\end{enumerate}
\end{Ex}

\section{Main Theorem}

In this section, we 
state our main theorem, 
and give  
a table of quandles of cyclic type 
with cardinality 
up to $12$. 
The following notations will be used 
throughout the remaining of this paper: 
\begin{itemize}
\item
$X := \{ 1 , 2 , \ldots , n \}$ with $n \geq 3$, 
\item
$S_n$ denotes 
the symmetry group of order $n$, 
\item
$(S_n)_{n-1} := \{ \sigma \in S_n \mid 
\mbox{$\sigma$ is a cyclic permutation of order $n-1$} \}$,   
\end{itemize}

\begin{Def} 
We denote by $C_n^{\#}$ the set of all quandle structures of cyclic type on $X$, 
that is, 
\begin{align*}
C_n^{\#} := \{ s : X \to (S_n)_{n-1} \mid 
\mbox{$s$ satisfies (S1), (S3)} \} . 
\end{align*}
(Note that every $s \in C_n^{\#}$ automatically satisfies (S2).) 
We denote by $C_{n}$  the set of isomorphism classes $[s]$ of $s \in C_{n}^{\#}$. 
\end{Def}

Consider the inclusion map from $C_n^{\#}$ 
into the set of quandles of cyclic type with cardinality $n$. 
This induces a bijection from $C_{n}$ to the set of 
isomorphism classes of quandles of cyclic type with cardinality $n$. 

\begin{Def}\label{def:F}
Let $s_{1} := (2 3 \cdots n)$. 
We denote by $F_{n}$ the set of $s_2 \in (S_n)_{(n-1)}$ 
satisfying the following two conditions: 
\begin{enumerate}
\item[(F1)] 
$s_2(2) = 2$, and
\item[(F2)] 
$\{ s_2^{m} s_{1} s_2^{-m} \mid m=1, 2, \dots , n-2 \} = 
\{ s_{1}^{m} s_{2} s_{1}^{-m} \mid m=1, 2, \dots, n-2 \}$. 
\end{enumerate}
\end{Def}

Recall that $(2 3 \cdots n)$ denotes the cyclic permutation. 
The following is the main theorem of this paper, 
which gives a one-to-one correspondence between $C_n$ and $F_n$. 

\begin{Thm}\label{thm:AnDn} 
Let $s_{1}:=(2 3 \cdots n)$, 
$s_2 \in F_n$, 
and define  
$\varphi(s_2) : X \to \mathrm{Map}(X,X)$ 
by 
\begin{align*}
(\varphi(s_2))_i := \left\{ 
\begin{array}{l@{\ \ \,}l@{\,}} 
s_{1} & (i=1) , \\ 
s_{2} & (i=2) , \\ 
s_{1}^{i-2} \circ s_{2} \circ s_{1}^{-i+2} & (i \in \{ 3, \ldots, n \}) . 
\end{array}
\right. 
\end{align*}
Then one has $\varphi(s_2) \in C^\#_n$, and hence give 
a map $\varphi : F_{n} \to C_{n}^{\#}$. 
This induces a bijection from $F_{n}$ onto $C_{n}$ 
by composing with the natural projection from $C_{n}^{\#}$ onto $C_{n}$. 
\end{Thm}

The proof of this theorem will be given in the next section. 
In the remaining of this section, 
we provide a table of quandles of cyclic type with cardinality up to $12$. 
For the classification, we have only to determine the set $F_n$. 

\begin{Prop}\label{pro:exa} 
We have $F_{3} = \{ (13) \}$ and $F_{4} = \{ (143) \}$. 
\end{Prop}

\begin{proof} 
The basic strategy is the following. 
First of all, we list up all elements in $(S_n)_{(n-1)}$ satisfying (F1). 
These elements are called the \textit{candidates} for simplicity. 
We then check whether each candidate satisfies (F2) or not. 

In the case of $n=3$, the only candidate is $s_{2}=(13)$. 
One can easily see that 
\begin{align}
s_2s_{1} s_2^{-1} = (12) =s_{1}s_{2} s_{1}^{-1} . 
\end{align} 
Hence $s_2$ satisfies (F2).
This proves the first assertion. 

In the case of $n=4$, there are two candidates, $(143)$ and $(134)$. 
For $s_2:=(143)$, we have 
\begin{align*} 
s_1 s_2 s_1^{-1}=(124), \quad 
s_1^2 s_2 s_1^{-2}=(132), \\
s_2 s_1 s_2^{-1}=(132), \quad 
s_2^2 s_1 s_2^{-2}=(124).
\end{align*} 
Thus $s_2=(143)$ satisfies (F2). 
On the other hand, $s_2:=(134)$ does not satisfy (F2). 
In fact, $s_2 s_1 s_2^{-1}=(124)$ is not an element of 
\begin{align}
\{s_1^{m}s_2 s_1^{-m} \mid m=1, 2\}=\{(142), (123)\} . 
\end{align}
This completes the proof of the second assertion. 
\end{proof}

When $n=3$, the quandle corresponding to $s_2 = (13)$ 
is the dihedral quandle with cardinality $3$. 
When $n=4$, the quandle corresponding to $s_2 = (143)$ 
is the tetrahedron quandle.  
When $n \geq 5$, the following lemma is useful to examine 
whether each candidate satisfies (F2) or not. 

\begin{Lem}\label{lemF}
Let $s_2 \in F_n$, and assume that $m \in \Z$ satisfies $s_2(1) = s_1^m(2)$. 
Then we have 
\begin{align*}
s_1^m s_2 s_1^{-m} = s_2 s_1 s_2^{-1} . 
\end{align*}
\end{Lem}

\begin{proof}
Since $s_2\in F_n$ satisfies (F2), 
there exists $l \in \{ 1, 2, \ldots, n-2 \}$ such that 
\begin{align}
s_1^m s_2 s_1^{-m} = s_2^l s_1 s_2^{-l} . 
\end{align}
Note that $s_1^m s_2 s_1^{-m}$ is a cyclic permutation of order $n-1$, 
having the unique fixed point $s_1^m(2)$. 
Similarly, $s_2^l s_1 s_2^{-l}$ has the unique fixed point $s_2^l(1)$. 
Hence, combining with the assumption, one has 
\begin{align}
s_2(1) = s_1^m(2) = s_2^l(1) . 
\end{align}
Since $s_2 \in (S_n)_{n-1} $ 
and it satisfies (F1), we conclude 
\begin{align}
l=1. 
\end{align} 
This completes the proof. 
\end{proof}

The above lemma is useful to determine the set $F_n$ for any $n$. 
Here we apply it to the case of $n=5$. 

\begin{Prop}\label{pro:exam2} 
We have
$F_{5} = \{ (1354) , (1435) \}$.  
\end{Prop}

\begin{proof}
As for the set $F_5$, there are six candidates, 
\begin{align}
s_2 = (1345), \ (1354), \ (1435), \ (1453), \ (1534), \ (1543) . 
\end{align}
One can directly see that $(1354)$ and $(1435)$ satisfy (F2). 
We omit the proof for these two cases. 

We here show that the remaining four candidates do not satisfy (F2). 
For the proof, we use Lemma~\ref{lemF}. 
In fact, we determine $m$ satisfying $s_{1}^{m}(2) = s_2(1)$, 
and show that $s_{1}^m s_2 s_{1}^{-m}\neq s_2 s_1 s_2^{-1}$. 

Case (1): $s_2:=(1345)$. 
Let $m=1$. 
Then one has $s_1^m(2) = 3 = s_2(1)$ and 
\begin{align}
s_{1}^m s_2 s_{1}^{-m}=(1452) 
\neq (1245) = s_2 s_1 s_2^{-1}. 
\end{align}

Case (2): $s_2:=(1453)$. 
Let $m=2$. 
Then one has $s_1^m(2) = 4 = s_2(1)$ and
\begin{align}
s_{1}^m s_2 s_{1}^{-m}=(1235)
\neq (1532)=s_2 s_1 s_2^{-1}.
\end{align}

Case (3): $s_2:=(1543)$. 
Let $m=3$. 
Then one has $s_1^m(2) = 5 = s_2(1)$ and
\begin{align}
s_{1}^m s_2 s_{1}^{-m}=(1432) 
\neq (1342)=s_2 s_1 s_2^{-1}.
\end{align}

Case (4): $s_2:=(1534)$. 
Let $m=3$. 
Then one has $s_1^m(2) = 5 = s_2(1)$ and
\begin{align}
s_{1}^m s_2 s_{1}^{-m}=(1423) 
\neq (1324)=s_2 s_1 s_2^{-1}.
\end{align}

We have thus proved that these four candidates do not satisfy (F2). 
\end{proof}

By the same arguments, 
we have determined $F_n$ with $n \leq 12$. 
We omit the proof since it is long and 
the arguments are exactly same as above,  
and we have used some computer programs for calculations. 
The results are summarized in Table~\ref{upto12table}, 
which gives a classification of quandles of cyclic type with 
cardinality up to $12$. 
Note that $\#F_n$ denotes the cardinality of $F_n$. 

\begin{table}
\begin{tabular}{|c|c|c|}
\hline
$n$ & $\#F_n$ & $F_{n}$ \\
\hline
$3$&1&$\{(1\ 3)\}$ \\
\hline
$4$&1&$\{(1\ 4\ 3)\}$ \\
\hline
$5$&2&$\{(1\ 3\ 5\ 4), (1\ 4\ 3\ 5)\}$ \\
\hline
$6$&0&$\emptyset$ \\
\hline
$7$&2&$\{(1\ 7\ 4\ 6\ 5\ 3), (1\ 7\ 5\ 4\ 6\ 3)\}$ \\
\hline
$8$&2&$\{(1\ 5\ 8\ 3\ 7\ 6\ 4), (1\ 7\ 5\ 4\ 8\ 3\ 6)\}$ \\
\hline
$9$&2&$\{(1\ 4\ 3\ 8\ 6\ 9\ 5\ 7), (1\ 5\ 7\ 3\ 6\ 4\ 9\ 8)\}$ \\
\hline
$10$&0&$\emptyset$ \\
\hline
$11$&4&$\{(1\ 3\ 6\ 8\ 4\ 11\ 5\ 10\ 9\ 7), (1\ 4\ 3\ 7\ 10\ 5\ 11\ 9\ 6\ 8), $ \\
$$&$$&$(1\ 6\ 8\ 5\ 3\ 9\ 4\ 7\ 11\ 10), (1\ 7\ 5\ 4\ 9\ 3\ 10\ 6\ 8\ 11)\}$ \\
\hline
$12$&0&$\emptyset$ \\
\hline
\end{tabular}
\caption{Quandles of cyclic type with cardinality up to $12$} 
\label{upto12table}
\end{table}

We note that Table~\ref{upto12table} 
agrees with some previously known results (\cite{Hayashi, Lopes-Roseman, Tamaru}) 
mentioned in Introduction. 
By looking at these classification lists, we conjecture the following. 

\begin{Con}
Let $n \geq 3$. 
Then, there exists a quandle of cyclic type with cardinality $n$ 
if and only if $n$ is a power of a prime number. 
\end{Con}

\section{Proof of Theorem~\ref{thm:AnDn}} 

In this section, we prove Theorem~\ref{thm:AnDn}, 
which gives a bijection from $F_n$ onto $C_n$. 
For the proof, we define auxiliary sets $E_n$ and $D_n$, and construct bijections 
\begin{align} 
g_3 : F_n \to E_n , \quad 
g_2 : E_n \to D_n , \quad 
g_1 : D_n \to C_n . 
\end{align}

\subsection{A bijection from $D_n$ onto $C_n$}

In this subsection, 
we define a set $D_n$, and construct a bijection from $D_n$ onto $C_n$. 
Recall that $X := \{ 1, \ldots, n \}$, and 
$(S_{n})_{n-1}$ is the subset of $S_{n}$ consisting of all cyclic permutations 
of order $n-1$. 
Two subsets $\Sigma$, $\Sigma^\prime \subset S_n$ are said to be 
\textit{conjugate} 
if there exists $g \in S_n$ such that $g^{-1} \Sigma g = \Sigma^\prime$. 

\begin{Def}
We denote by $D_n^\#$ the set of $\Sigma \subset (S_{n})_{n-1}$ satisfying 
\begin{enumerate} 
\item[(D1)] 
$\forall s \in \Sigma$, $s^{-1} \Sigma s \subset \Sigma$, and 
\item[(D2)] 
$\forall x \in X$, $\exists! s \in \Sigma$ : $s(x)=x$. 
\end{enumerate} 
We also denote by $D_n$ the set of 
conjugacy classes 
$[\Sigma]$ of $\Sigma \in D_n^\#$. 
\end{Def}

First of all we study $D_n^\#$. 
Note that Conditions~(D1) and (D2) are preserved by conjugation. 
Namely, if $\Sigma \in D_n^\#$ and $\Sigma$ is conjugate to $\Sigma^\prime$, 
then one has $\Sigma^\prime \in D_n^\#$. 
Furthermore, the following lemma yields that 
every $\Sigma \in D_n^\#$ satisfies $\# \Sigma = n$. 

\begin{Lem}
\label{lem:sigmaquandle-cardinality}
Let $\Sigma \in D_n^\#$. 
For each $x \in X$, 
denote by $s^\Sigma_{x} \in \Sigma$ the 
unique element with $s^\Sigma_{x}(x)=x$.  
Then, the obtained map $s^\Sigma : X \to \Sigma$ is bijective. 
\end{Lem}

\begin{proof}
We show that $s^\Sigma$ is surjective. 
Take any $s \in \Sigma$. 
Since $s \in (S_n)_{n-1}$, there exists $x \in X$ such that $s(x) = x$. 
Then, the uniqueness in (D2) shows 
$s = s_x^\Sigma$. 

We next show that $s^\Sigma$ is injective. 
Let $x, y \in X$ and assume that $s_x^\Sigma = s_y^\Sigma$. 
One knows $s_x^\Sigma (x) = x$ by definition. 
Thus $x$ is the unique fixed point of $s_x^\Sigma \in (S_n)_{n-1}$. 
Similarly, $y$ is the unique fixed point of $s_y^\Sigma$. 
This concludes $x=y$. 
\end{proof}

The aim of this subsection is to construct a bijection from $D_n$ onto $C_n$. 
We here see that $s^\Sigma$ defines a map from $D_n^\#$ onto $C_n^\#$. 
Recall that $\Sigma \subset (S_n)_{n-1}$. 

\begin{Lem}
\label{lem:sigmaquandle1}
The above defined map $s^\Sigma : X \to (S_n)_{n-1}$ 
satisfies $s^\Sigma \in C_n^\#$, 
that is, $(X, s^\Sigma)$ is a quandle of cyclic type.  
\end{Lem} 

\begin{proof}
By definition, $s^{\Sigma}$ satisfies (S1). 
Hence we have only to show (S3). 
Take  any $y, z \in X$. 
Condition~(D1) yields that 
\begin{align}
s^\Sigma_{z} \circ s^\Sigma_{y} \circ (s^{\Sigma}_{z})^{-1} \in \Sigma . 
\end{align}
On the other hand, one has 
\begin{align}
s^{\Sigma}_{z}\circ s^{\Sigma}_{y}\circ (s^{\Sigma}_{z})^{-1}(s^{\Sigma}_z(y)) 
= s^{\Sigma}_{z}\circ s^{\Sigma}_{y}(y)=s^{\Sigma}_z(y) . 
\end{align}
Therefore, from the uniqueness 
in (D2), one has 
\begin{align}
s^\Sigma_{z}\circ s^\Sigma_{y}\circ (s^{\Sigma}_{z})^{-1} 
= s^\Sigma_{s^\Sigma_{z}(y)} . 
\end{align}
This proves (S3), which completes the proof. 
\end{proof} 

One thus has obtained a map from $D_n^\#$ to $C_n^\#$. 
For the later use, we here show that this map is surjective. 

\begin{Lem}
\label{lem:sigmaquandle-add}
The following map is surjective$:$ 
\begin{align*}
\bar{g}_1 : D_n^\# \to C_n^\# : \Sigma \mapsto s^\Sigma . 
\end{align*}
\end{Lem} 

\begin{proof}
Take any $s \in C_n^\#$. 
Let us put 
\begin{align} 
\Sigma := \{ s_{x} \mid x \in X \} \subset (S_n)_{n-1} . 
\end{align} 
We have only to prove that $\Sigma \in D_n^\#$ and $\bar{g}_1(\Sigma) = s$. 

We show that $\Sigma$ satisfies (D1). 
Take any $s_x, s_y \in \Sigma$. 
Since $s_x^{-1}$ is an automorphism, 
one has 
\begin{align} 
s_{x}^{-1} \circ s_{y} \circ s_{x} = s_{s_x^{-1} (y)} 
\in \Sigma . 
\end{align} 
This proves $s_{x}^{-1} \Sigma s_{x} \subset \Sigma$. 

We next show that $\Sigma$ satisfies (D2). 
Take any $x \in X$. 
Since $s$ satisfies (S1), 
$s_x \in \Sigma$ satisfies $s_x(x) = x$. 
This proves the existence. 
Next assume that $s_y(x) = x$. 
Since $s \in C_n^\#$, one has $s_y \in (S_n)_{n-1}$. 
Hence $x$ is the unique fixed point of $s_y$. 
Thus (S1) yields that $x=y$, which proves the uniqueness. 

We have 
proved $\Sigma \in D_n^\#$. 
Furthermore, by the definition of $\bar{g}_1$, it is easy to see that 
$\bar{g}_1(\Sigma) = s$. 
This completes the proof. 
\end{proof} 

We here define a map from $D_n$ to $C_n$. 
Recall that $[\Sigma]$ denotes the 
conjugacy class 
of $\Sigma \in D_n^\#$, 
and $[s]$ denotes the isomorphism class of $s \in C_n^\#$. 

\begin{Lem}
\label{lem:sigmaquandle2}
The following map is well-defined$:$ 
\begin{align*}
g_1 : D_n \to C_n : [\Sigma] \mapsto [s^\Sigma] . 
\end{align*}
\end{Lem} 

\begin{proof}
Let $\Sigma , \Sigma^\prime \in D_n^\#$, 
and assume that $[\Sigma] = [\Sigma^\prime]$. 
Hence there exists $g \in S_{n}$ such that $\Sigma=g^{-1}\Sigma^\prime g$. 
In order to show $[s^\Sigma] = [s^{\Sigma^\prime}]$, 
it is enough to prove that the following map is a quandle isomorphism: 
\begin{align}
g : (X, s^{\Sigma}) \to (X, s^{\Sigma^\prime}) . 
\end{align}
This is obviously bijective. 
We show that 
$g$ is a quandle homomorphism. 
Take any $x \in X$. 
By definition, one has
\begin{align} 
s^{\Sigma^\prime}_{g(x)} (g(x)) = g(x) . 
\end{align} 
This means that 
\begin{align} 
g^{-1} \circ s^{\Sigma^\prime}_{g(x)} \circ g (x) = x . 
\end{align} 
On the other hand, one has 
\begin{align} 
g^{-1} \circ s^{\Sigma^\prime}_{g(x)} \circ g \in g^{-1} \Sigma^\prime g 
= \Sigma . 
\end{align} 
Hence, from the uniqueness 
in (D2), we have 
\begin{align} 
g^{-1} \circ s^{\Sigma^\prime}_{g(x)} \circ g = s^{\Sigma}_x . 
\end{align} 
This proves that $g$ is a quandle homomorphism. 
\end{proof}

We now show that the above defined map $g_1$ is bijective. 
The following is the main result of this subsection. 

\begin{Prop}\label{pro:Dn} 
The map $g_1 : D_n \to C_n$ is bijective. 
\end{Prop}

\begin{proof}
One knows that $g_1$ is surjective, 
since so is $\bar{g}_1$ from Lemma~\ref{lem:sigmaquandle-add}. 
It remains to show that $g_{1}$ is injective. 
Let $[\Sigma] , [\Sigma^\prime] \in D_n$, 
and assume that $g_1 ([\Sigma]) = g_1 ([\Sigma^\prime])$. 
By definition, one has $[s^{\Sigma}] = [s^{\Sigma^\prime}]$, 
that is, there exists a quandle isomorphism 
\begin{align} 
g : (X, s^{\Sigma}) \rightarrow (X, s^{\Sigma^\prime}) . 
\end{align} 
Since $g$ is bijective, one has $g \in S_{n}$. 
Since $g$ is a homomorphism, one has for any $x \in X$ that 
\begin{align}
s^\Sigma_{x} 
= g^{-1} \circ s^{\Sigma^\prime}_{g(x)} \circ g 
\in g^{-1} \Sigma^\prime g . 
\end{align} 
This proves $\Sigma \subset g^{-1} \Sigma^\prime g$. 
Recall that $\# \Sigma = n = \# \Sigma^\prime$ holds 
from Lemma~\ref{lem:sigmaquandle-cardinality}. 
Therefore, we have $\Sigma = g^{-1} \Sigma^\prime g$, 
and thus $[\Sigma] = [\Sigma^\prime]$. 
This concludes that $g_1$ is injective. 
\end{proof} 

\subsection{A bijection from $E_n$ onto $D_n$}

In this subsection, we define a set $E_n$, 
and construct a bijection from $E_n$ onto $D_n$. 
We denote by 
\begin{align}
S_{n, (1, 2)} := \{ u \in S_{n} \mid u(1)=1 , \, u(2) = 2 \} . 
\end{align}
Two elements 
$(u_1, u_2)$, $(v_1, v_2) \in (S_{n})_{n-1} \times (S_{n})_{n-1}$ 
are said to be \textit{$S_{n,(1,2)}$-conjugate} 
if $(u_1, u_2)= (w^{-1} v_1 w, w^{-1} v_2 w)$ 
for some $w \in S_{n, (1, 2)}$. 

\begin{Def}
We denote by $E_n^\#$ the set of 
$(u_{1}, u_{2}) \in (S_{n})_{n-1} \times (S_{n})_{n-1}$ satisfying 
\begin{enumerate}
\item[(E1)] 
$u_{1}(1)=1$, $u_{2}(2)=2$, and
\item[(E2)] 
$\{ u_{1}^{m} u_{2} u_{1}^{-m} \mid m=1, 2, \dots , n-2 \} 
= \{u_{2}^{m} u_{1} u_{2}^{-m} \mid m=1, 2, \dots , n-2 \}$. 
\end{enumerate}
We also denote by $E_n$ the set of 
$S_{n, (1,2)}$-conjugacy classes 
$[(u_1, u_2)]$ of $(u_1, u_2) \in E_n^\#$.  
\end{Def}

First of all, we construct a map from $E_{n}^\#$ to $D_{n}^\#$. 

\begin{Lem}
Let $(u_1, u_2)\in E_{n}^\#$. 
Then one has
\begin{align}
\Sigma_{(u_1, u_2)} := 
\{ u_{1}, u_{2} \} \cup \{ u_{1}^{m} u_{2} u_{1}^{-m} \mid 
m = 1 , 2 , \dots , n-2 \} \in D_n^\# . 
\end{align}
\end{Lem}

\begin{proof} 
We have only to show that $\Sigma_{(u_1, u_2)}$ satisfies (D1) and (D2). 
In order to show (D1), it is enough to prove 
\begin{align} 
\label{eq:EtoD} 
u_1^{-1} \Sigma_{(u_1, u_2)} u_1 \subset \Sigma_{(u_1, u_2)}, \quad 
u_2^{-1} \Sigma_{(u_1, u_2)} u_2 \subset \Sigma_{(u_1, u_2)} . 
\end{align} 
Note that $u_1$ has order $n-1$. 
Then one has 
\begin{align} 
\begin{split} 
u_1^{-1} u_1 u_1 & = u_1 
\in \Sigma_{(u_1, u_2)} , \\ 
u_1^{-1} u_2 u_1 & = u_1^{n-2} u_2 u_1^{-(n-2)} 
\in \Sigma_{(u_1, u_2)} , \\ 
u_1^{-1} (u_{1}^{m} u_{2} u_{1}^{-m}) u_1 
& = u_{1}^{m-1} u_{2} u_{1}^{-(m-1)} 
\in \Sigma_{(u_1, u_2)} 
\quad (\mbox{for $m = 1, \ldots , n-2$}) . 
\end{split} 
\end{align} 
This proves the former claim of (\ref{eq:EtoD}). 
On the other hand, (E2) 
yields that 
\begin{align}
\Sigma_{(u_1, u_2)} = 
\{ u_{1} , u_{2} \} \cup 
\{ u_{2}^{m} u_{1} u_{2}^{-m} \mid m = 1 , 2 , \dots , n-2 \} . 
\end{align}
Hence, a similar calculation 
proves the latter claim of (\ref{eq:EtoD}). 

We next show (D2). 
Take any $x \in X$. 
If $x = 1, 2$, then it is fixed by $u_1, u_2 \in \Sigma_{(u_1, u_2)}$, 
respectively. 
Assume that $x \neq 1,2$. 
By (E1) and $u_1 \in (S_n)_{n-1}$, 
there exists $m \in \{ 1, \ldots , n-2 \}$ such that $x = u_{1}^{m}(2)$. 
Then one has
\begin{align}
u_{1}^{m} u_{2} u_{1}^{-m} (x) 
= u_{1}^{m} u_{2} u_{1}^{-m} (u_{1}^{m}(2)) 
= u_{1}^{m} u_{2} (2)
= u_{1}^{m} (2) = x . 
\end{align}
This completes the proof of the existence. 
On the other hand, by definition one has $\# \Sigma_{(u_1, u_2)} \leq n$. 
This shows the uniqueness. 
\end{proof}

This lemma constructs a map from $E_{n}^\#$ to $D_{n}^\#$. 
We next show that this induces a map from $E_{n}$ to $D_{n}$. 

\begin{Lem} 
The following map is well-defined$:$
\begin{align*}
g_{2} : E_{n} \rightarrow D_{n} : [(u_{1}, u_{2})] \mapsto [\Sigma_{(u_1, u_2)}] . 
\end{align*}
\end{Lem}

\begin{proof}
Let $[(u_{1}, u_{2})], [(u^\prime_{1}, u^\prime_{2})]\in E_n$, 
and assume that $[(u_{1}, u_{2})]=[(u^\prime_{1}, u^\prime_{2})]$. 
Then there exists $w\in S_{n, (1,2)}$ 
such that 
\begin{align}
u_{1} =w^{-1} u^\prime_{1} w , \quad 
u_{2} = w^{-1} u^\prime_{2} w . 
\end{align}
Furthermore, for every $m \in \{ 1, \ldots , n-2 \}$, one 
has 
\begin{align} 
w^{-1} ({u^\prime_{1}}^{m} u^\prime_{2} {u^\prime_{1}}^{-m}) w 
= (w^{-1} {u^\prime_{1}} w)^m 
(w^{-1} u^\prime_{2} w) 
(w^{-1} {u^\prime_{1}} w)^{-m} 
= u_{1}^{m} u_{2} u_{1}^{-m} . 
\end{align} 
We thus have 
$w^{-1} \Sigma_{(u^\prime_1, u^\prime_2)} w \subset \Sigma_{(u_1, u_2)}$. 
This proves 
\begin{align}
w^{-1} \Sigma_{(u^\prime_1, u^\prime_2)} w = \Sigma_{(u_1, u_2)} , 
\end{align}
since 
$\Sigma_{(u^\prime_1, u^\prime_2)} , \Sigma_{(u_1, u_2)} \in D_n^\#$, 
and hence 
$\# \Sigma_{(u^\prime_1, u^\prime_2)} = n = \# \Sigma_{(u_1, u_2)}$ 
by Lemma~\ref{lem:sigmaquandle-cardinality}. 
This completes the proof of 
$[\Sigma_{(u_1, u_2)}]=[\Sigma_{(u^\prime_1, u^\prime_2)}]$. 
\end{proof}

The aim of this subsection is to prove that $g_2$ is bijective, 
by constructing the inverse map. 
For this purpose, 
we construct a map 
from $D_n^\#$ to $E_n^\#$. 
Recall that we have a map 
\begin{align} 
\bar{g}_1 : D_n^\# \to C_n^\# : \Sigma \mapsto s^\Sigma . 
\end{align} 

\begin{Lem}
\label{lem:DtoE}
Let $\Sigma \in D_n^\#$. 
Then one has $(s^\Sigma_1, s^\Sigma_2) \in E_n^\#$. 
\end{Lem}

\begin{proof}
For simplicity of the notations, 
we put $s_x := s^\Sigma_x$ for each $x \in X$. 
By definition, $(s_1, s_2)$ obviously satisfies (E1).  
We have only to show (E2). 

First of all, we claim that 
\begin{align}
\label{eq:claim5-10}
\{ s_{1}^{m} s_{2} s_{1}^{-m} \mid m = 1, 2, \dots , n-2 \} 
= \{s_x \mid x = 3, 4, \dots , n \} . 
\end{align} 
Let $m \in \{ 1, 2, \dots n-2 \}$. 
Since $\Sigma$ satisfies (D1), one has
\begin{align}
s_{1}^{m} s_{2} s_{1}^{-m} \in s_{1}^{m} \Sigma s_{1}^{-m} \subset \Sigma . 
\end{align}
Thus, it follows from 
$s_{1}^{m} s_{2} s_{1}^{-m} (s_{1}^{m}(2))= s_{1}^{m}(2)$ 
and the uniqueness in (D2) that 
\begin{align}
s_{1}^{m} s_{2} s_{1}^{-m} = s_{s_{1}^{m}(2)} . 
\end{align} 
Since $s_1(1) = 1$ and $s_1 \in (S_n)_{n-1}$, one has 
\begin{align}
\{ s_{1}^{m}(2) \mid m = 1, 2, \ldots, n-2 \} = 
\{ 3, 4, \ldots , n \} . 
\end{align}
This completes the proof of the claim. 

By the same argument, one can see that 
\begin{align} 
\{ s_{2}^{m} s_{1} s_{2}^{-m} \mid m = 1, 2, \dots , n-2 \} 
= \{ s_{ x } \mid x 
= 3, 4, \dots , n \} .  
\end{align} 
This and the above claim prove (D2). 
\end{proof}

The above lemma gives a map from $D_n^\#$ to $E_n^\#$. 
We next show that this map induces a map from $D_n$ to $E_n$. 
\begin{Lem}\label{lem:Dn2} 
The following map is well-defined$:$ 
\begin{align}
f_{2} : D_{n} \rightarrow E_{n} : [\Sigma] \mapsto [(s^\Sigma_1 , s^\Sigma_2)] . 
\end{align}
\end{Lem}

\begin{proof}
Let $[\Sigma], [\Sigma^\prime]\in D_{n}$, 
and assume 
that $[\Sigma]=[\Sigma^\prime]$. 
By definition, 
there exists $g \in S_n$ such that $\Sigma = g^{-1} \Sigma^\prime g$. 
It then follows from Lemma~\ref{lem:sigmaquandle2} that 
\begin{align}
g : (X, s^{\Sigma}) \to (X, s^{\Sigma^\prime}) 
\end{align}
is a quandle isomorphism. 
Note that $(X, s^{\Sigma^\prime})$ is of cyclic type, 
and hence two-point homogeneous. 
Therefore, since $g(1) \neq g(2)$, 
there exists 
$h \in \mathrm{Inn}(X, s^{\Sigma^\prime})$ 
such that 
\begin{align} 
(h \circ g (1), h \circ g (2)) = (1, 2) . 
\end{align} 
This yields
$h \circ g \in S_{n,(1,2)}$. 
Note that $h \circ g$ is a quandle isomorphism 
from $(X, s^{\Sigma})$ onto $(X, s^{\Sigma^\prime})$. 
Thus one has 
\begin{align} 
\begin{split}
(h \circ g) \circ 
s^\Sigma_1 
\circ (h \circ g)^{-1} = 
s^{\Sigma^\prime}_{h \circ g(1)} = s^{\Sigma^\prime}_{1} , \\ 
(h \circ g) \circ s^\Sigma_2 \circ (h \circ g)^{-1} 
= s^{\Sigma^\prime}_{h \circ g(2)} = s^{\Sigma^\prime}_{2} . 
\end{split} 
\end{align} 
This completes the proof of 
$[(s^\Sigma_1 , s^\Sigma_2)] 
= [(s^{\Sigma^\prime}_1 , s^{\Sigma^\prime}_2)]$. 
\end{proof}

By showing that $f_2$ is the inverse map of $g_2$, 
we have the following main result of this subsection. 

\begin{Prop}
The map $g_2 : E_{n} \rightarrow D_{n}$ is bijective. 
\end{Prop}

\begin{proof}
We show that $f_2$ is the inverse map of $g_2$. 
It is clear that the composition 
$f_{2}\circ g_{2}$ is the identity mapping. 
Consider $g_{2} \circ f_{2} : D_n \to D_n$, 
and take any $[\Sigma] \in D_{n}$. 
Then one has $f_2 ([\Sigma]) = [(s^\Sigma_1 , s^\Sigma_2)]$. 
One also has $g_2 \circ f_2 ([\Sigma]) = [\Sigma^\prime]$, where 
\begin{align}
\Sigma^\prime := \{ s^\Sigma_1 , s^\Sigma_2 \} \cup 
\{ (s^\Sigma_1)^m s^\Sigma_2 (s^\Sigma_1)^{-m} \mid m = 1, \ldots , n-2 \} . 
\end{align}
Since $s^\Sigma$ is a quandle structure, 
one can see $\Sigma^\prime \subset \Sigma$. 
Thus we have $\Sigma^\prime = \Sigma$ for cardinality reason. 
This shows that $g_{2} \circ f_{2}$ is the identity mapping. 
\end{proof}

\subsection{A bijection from $F_n$ onto $E_n$} 

We lastly construct a bijection 
from $F_n$ onto $E_n$. 
Let $s_1 := (23 \cdots n)$, and 
recall that $F_n$ is the set of $s_2 \in (S_n)_{n-1}$ satisfying (F1) and (F2). 

\begin{Prop}
The following map is bijective$:$ 
\begin{align*}
g_{3} : F_{n} \rightarrow E_{n} : s_{2} \mapsto [(s_{1},s_{2})] . 
\end{align*}
\end{Prop}

\begin{proof}
We show that $g_3$ is surjective. 
Take any $[(u_{1},u_{2})] \in E_{n}$. 
Since $u_1 \in (S_n)_{n-1}$ and $u_1(1) = 1$, 
we can write $u_1 = (2 a_{3} a_{4} \cdots a_{n})$. 
Let us define $g \in S_{n,(1,2)}$ by 
\begin{align}
g:= \left( \begin{array}{ccccc}
1& 2& 3& \cdots & n\\
1& 2& a_{3}& \cdots & a_{n}
\end{array}
\right) . 
\end{align} 
An easy computation shows $g^{-1} \circ u_{1} \circ g = s_{1}$. 
Let $s_{2}:=g^{-1}\circ u_{2}\circ g$. 
Then $s_2$ obviously satisfies (F1). 
Furthermore, since $(u_1, u_2)$ satisfies (E2), 
one can see that $s_2$ satisfies (F2). 
We thus have $s_2 \in F_n$. 
This concludes that $g_{3}$ is surjective, since 
\begin{align} 
g_3(s_2) = [(s_1,s_2)] 
= [(g^{-1} \circ u_1 \circ g , g^{-1} \circ u_2 \circ g)] 
= [(u_1,u_2)] . 
\end{align} 

We show that $g_3$ is injective. 
Let $s_{2}$, $s_{2}^\prime \in F_{n}$, 
and suppose that $g_{3}(s_{2})=g_{3}(s_{2}^\prime)$. 
Hence there exists $h \in S_{n, (1,2)}$ such that 
\begin{align}
(s_{1}, s_{2}^\prime) 
= (h \circ s_{1} \circ h^{-1} , h \circ s_{2} \circ h^{-1}) . 
\end{align} 
By definition one has $h(1)=1$ and $h(2)=2$. 
Then it follows from $h(2) = 2$ that 
\begin{align} 
3 = s_{1}(2) = h \circ s_{1} \circ h^{-1}(2) = h \circ s_{1} (2) = h(3) . 
\end{align} 
Similarly, this yields that 
\begin{align} 
4 = s_{1}(3) = h \circ s_{1} \circ h^{-1}(3) = h \circ s_{1} (3) = h(4) . 
\end{align} 
One can show 
inductively that $x = h(x)$ for any $x \in X$. 
This means that $h = \id$, 
and thus $s^\prime_2 = s_2$. 
This shows that 
$g_{3}$ is injective. 
\end{proof}

\subsection{Constructing quandles of cyclic type from $F_{n}$}

In the previous subsections, we have constructed 
the following bijections: 
\begin{align} 
g_3 : F_n \to E_n , \quad 
g_2 : E_n \to D_n , \quad 
g_1 : D_n \to C_n .  
\end{align} 
In this subsection, 
we describe $g_1 \circ g_2 \circ g_3 (s_2)$ for each $s_2 \in F_n$. 

Take any $s_2 \in F_n$. 
Recall that $s_1 := (2 3 \cdots n)$ and 
\begin{align}
\Sigma_{(s_1, s_2)} 
:= \{ s_{1} , s_2 \} \cup \{ s_{1}^{m} s_2  s_{1}^{-m} 
\mid m = 1 , 2 , \dots , n-2 \} 
\in D_n^\# . 
\end{align}
Then one has 
$g_2 \circ g_3(s) = [\Sigma_{(s_1, s_2)}]$. 
We put 
\begin{align} 
\varphi (s_2) := s^{\Sigma_{(s_1, s_2)}} \in C_n^\# . 
\end{align} 
This means $g_1 \circ g_2 \circ g_3 (s_2) = [\varphi (s_2)]$. 
Note that 
$(\varphi (s_2))_i \in \Sigma_{(s_1, s_2)}$ 
is defined as the unique element fixing $i \in X$. 
This immediately yields 
\begin{align} 
(\varphi (s_2))_1 = s_1 , \quad (\varphi(s_2))_2 = s_2 . 
\end{align} 
Let $i \in \{ 3, \ldots , n \}$. 
Then one has $i = s_1^{i-2}(2)$, and hence 
\begin{align} 
s_1^{i-2} s_2 s_1^{-(i-2)} (i) = s_1^{i-2} s_2 (2) = s_1^{i-2} (2) = i . 
\end{align} 
This concludes that 
\begin{align} 
(\varphi (s_2))_i = s_1^{i-2} s_2 s_1^{-(i-2)} \quad 
(\mbox{for $i \in \{ 3, \ldots , n \}$}) , 
\end{align} 
which completes the proof of Theorem~\ref{thm:AnDn}.

\end{document}